\documentclass[11pt]{amsart}
\usepackage{amssymb,upref, enumerate}
\usepackage{tikz}
\usepackage{amsmath,amssymb,amscd,amsthm,amsxtra}
\usepackage{latexsym}

\usepackage{amsfonts}
\usepackage{verbatim}
\usepackage{pslatex}

\def\Xint#1{\mathchoice
{\XXint\displaystyle\textstyle{#1}}%
{\XXint\textstyle\scriptstyle{#1}}%
{\XXint\scriptstyle\scriptscriptstyle{#1}}%
{\XXint\scriptscriptstyle\scriptscriptstyle{#1}}%
\!\int}
\def\XXint#1#2#3{{\setbox0=\hbox{$#1{#2#3}{\int}$}
\vcenter{\hbox{$#2#3$}}\kern-.5\wd0}}

\def\dashint{\Xint-}

\newcommand{\ra}{\rightarrow}

\newcommand{\bey}{\begin{eqnarray*}}
\newcommand{\eey}{\end{eqnarray*}}
\newcommand{\ba}{\begin{align}}
\newcommand{\ea}{\end{align}}
\newcommand{\bea}{\begin{align*}}
\newcommand{\ena}{\end{align*}}
\newcommand{\be}{\begin{equation}}
\newcommand{\ee}{\end{equation}}
\newcommand{\R}{\mathbb R}
\newcommand{\rn}{\mathbb R^n}

\newcommand{\C}{\mathbb C}

\newcommand{\I}{\mathcal I}

\newcommand{\M}{\mathcal M }

\newcommand{\ep}{\epsilon}

\newcommand{\bc}{\begin{center}}
\newcommand{\ec}{\end{center}}

\newcommand{\al}{\alpha}

\newtheorem{theorem}{Theorem}[section]

\newtheorem*{theorema*}{Theorem A}
\newtheorem*{theoremb*}{Theorem B}
\newtheorem*{theoremc*}{Theorem C}

\theoremstyle{definition}

\newtheorem{question}[theorem]{Question}

\theoremstyle{remark}

\newcommand{\cz}{Calder\'on-Zygmund}

\begin{document}



\title[Compactness of bilinear commutators]{Compactness properties of commutators of bilinear fractional integrals}

\date{\today}

\author[\'A. B\'enyi]{\'Arp\'ad B\'enyi}

\address{%
Department of Mathematics \\
Western Washington University\\
516 High Street\\
Bellingham, WA 98225, USA}
\email{arpad.benyi@wwu.edu}

\author[W. Dami\'an]{Wendol\'in Dami\'an}
\address{%
Departamento de An\'alisis Matem\'atico, Facultad de Matem\'aticas \\
Universidad de Sevilla\\
41080 Sevilla, Spain}
\email{wdamian@us.es}

\author[K. Moen]{Kabe Moen}
\address{%
Department of Mathematics\\
University of Alabama \\
Tuscaloosa, AL 35487, USA}
\email{kabe.moen@ua.edu}

\author[R. H. Torres]{Rodolfo H. Torres}
\address{%
Department of Mathematics\\
University of Kansas\\
Lawrence, KS 66045, USA}
\email{torres@math.ku.edu}

\thanks{\'A. B. partially supported by a grant from the Simons Foundation (No.~246024). K. M. and R.H. T. partially supported by NSF grants 1201504 and 1069015, respectively.}

\begin{abstract}
{Commutators of a large class of bilinear operators and multiplication by functions in a certain subspace of the space of functions of bounded mean oscillations  are shown to be jointly compact. Under a similar commutation, fractional integral versions of the bilinear Hilbert transform yield separately compact operators.}
\end{abstract}

\maketitle

\section{Introduction}

The smoothing effect of commutators of linear operators is nowadays a well known and very useful fact. For the purposes of this paper, ``smoothing'' will mean the improvement of boundedness to the stronger condition of compactness. A pilar for such considerations in linear setting is the work of Uchiyama \cite{Uch78}, where he showed that linear commutators of Calder\'on-Zygmund operators and pointwise multiplication with a symbol belonging to an appropriate subspace of the John-Nirenberg space $BMO$ are compact. Thus, indeed, these commutators behave better than just being bounded, a result earlier proved by Coifman, Rochberg and Weiss \cite{CRW}. Once compactness is established, one can derive a Fredholm alternative for equations with appropriate coefficients in all $L^p$ spaces with $1<p<\infty$, as in the work of Iwaniec and Sbordone \cite{IS}. Similarly, the theory of compensated compactness of Coifman, Lions, Meyer and Semmes \cite{CLMS} or the integrability theory of Jacobians, see, for example, the work of Iwaniec \cite{Iwa07}, owe a lot to the smoothing effect of commutators.

Bilinear commutators are naturally appearing operators in harmonic analysis, which leads to the equally relevant question about their smoothing behavior. For a bilinear operator $T$, and $b$ an appropriately smooth function, we will consider the following {\it bilinear commutators}:
\begin{align*}
[T, b]_1(f, g)&= T(bf,g)-bT(f,g),\\
[T, b]_2 (f, g)&=T(f,bg)-bT(f,g).
\end{align*}
The notion of compactness in bilinear setting goes back to Calder\'on foundational article \cite{Cal64}. Using the terminology in the work of B\'enyi and Torres \cite{BT}, we will be considering here the joint compactness (or simply compactness) and separate compactness of such bilinear operators.

Given three normed spaces $X, Y, Z$, a bilinear operator $T: X\times Y\to Z$ is said to be \emph{(jointly) compact} if the set $\{T(x, y): \|x\|, \|y\|\leq 1\}$ is precompact in $Z$. Writing $B_{1, X}$ for the closed unit ball in $X$, the definition of compactness specifically requires that if $\{(x_n,y_n)\}\subseteq B_{1,X}\times B_{1,Y}$, then the sequence $\{T(x_n,y_n)\}$ has a convergent subsequence in $Z$. Clearly, any compact bilinear operator $T$ is continuous.

We say that $T:X\times Y\rightarrow Z$ is \emph{compact in the first variable} if $T_y=T(\cdot, y): X\rightarrow Z$ is compact for all $y\in Y$. $T$ is called \emph{compact in the second variable} if $T_x=T(x, \cdot): Y\rightarrow Z$ is compact for all $x\in Y$. Finally, $T$ is called separately compact if $T$ is compact both in the first and second variable. While, in general, it is only true that separate compactness implies separate continuity, if we further consider one of the spaces $X$ or $Y$ to be Banach, the boundedness of $T$ follows from separate compactness as well. For more on these notions of compactness and their basic properties, we refer the interested reader to \cite{BT}.

Throughout this paper, the relevant space for the multiplicative symbols in our commutators will be a subspace of $BMO$, which we denote by $CMO$\footnote{The notion $CMO$ for this space is not uniformly used throughout the literature. See \cite{BT} for remarks and references about this notation.}. Recall that $BMO$ consists of all locally integrable functions $b$ with $\|b\|_{BMO}<\infty$, where
$\|b\|_{BMO}=\sup_{Q}\frac{1}{|Q|}\int_Q |b-b_Q|\,dx,$
with the supremum taken over all cubes $Q\in \rn$ and $b_Q=\frac{1}{|Q|}\int_Q b\,dx$ denoting the average of $b$ on $Q$. We define $CMO$ to be the closure of $C_c^\infty(\R^n)$ in the $BMO$ norm.

Consider for the moment $T$ to be a bilinear Calder\'on-Zygmund operator as defined by Grafakos and Torres \cite{GT}. For simplicity, we further assume that the kernels $K$ and $\nabla K$ satisfy the appropriate decay conditions in such theory. If $b, b_1, b_2\in BMO$, the bilinear commutators can be (formally) expressed in the form
\begin{align*}\label{formally}
[T, b]_1(f, g)(x)&=\int_{\mathbb R^n}\int_{\mathbb R^n} K(x, y, z)(b(y)-b(x))f(y)g(z)\, dydz,\\
[T, b]_2 (f, g)(x)&=\int_{\mathbb R^n}\int_{\mathbb R^n} K(x, y, z)(b(z)-b(x))f(y)g(z)\, dydz.\\
\end{align*}
The main result proved in \cite{BT} confirms that the smoothing effect of commutators of such operators with $CMO$ symbols is present in the bilinear setting as well; thus also extending Uchiyama's result mentioned before to bilinear commutators.  In fact, one has the following.

\begin{theorema*}
Let $T$ be a bilinear \cz\, operator.
If $b\in CMO$,  $1/p+1/q=1/r$, $1<p,q<\infty$ and  $1\leq r <\infty$, then, for $i=1, 2$, $[T, b]_i: L^p\times L^q\rightarrow L^r$ is compact.
\end{theorema*}

The proof of Theorem A, as well as the main results proved in this work and other compactness results in the literature,
make use of a known characterization of precompactness in Lebesgue spaces, known as the Fr\'echet-Kolmogorov-Riesz theorem; see, for example, Yosida's book \cite{Yo}.

\begin{theoremb*}
Let $1\leq r<\infty$. A subset $\mathcal K\subseteq L^r$ is compact if and only if the following three conditions are satisfied:
\begin{enumerate}[{\rm (a)}]
\item $\mathcal K$ is bounded in $L^r$;
\item $\lim_{A\ra \infty} \int_{|x|>A}|f(x)|^r\,dx=0$ uniformly for $f\in \mathcal K$;
\item $\lim_{t\ra 0}\|f(\cdot+t)-f\|_{L^r}=0$ uniformly for $f\in \mathcal K$.
\end{enumerate}
\end{theoremb*}

The goal of this paper is two fold. First, it aims to extend Theorem A to a larger class of bilinear operators, denoted by $\{T_\al\}_{\al>0}$, that has as limiting case for $\al= 0$ the bilinear Calder\'on-Zygmund operators. This is obtained in Theorem~2.1 below. Second, it  investigates a more singular version of these operators, whose limiting case is the bilinear Hilbert transform, and shows that the smoothing phenomenon, albeit weaker, is still present under commutation. This is achieved in Theorem 3.2.

\medskip

\noindent{\bf Acknowledgment.} This work was positively impacted by the interactions that occurred during  B\'enyi's and Torres'  stay at the Erwin Schr\"odinger Institute (ESI), Vienna, Austria, for the special semester on Modern Methods of Time-Frequency Analysis II. They wish to express their gratitude to the ESI and the organizers of the event for their support and warm hospitality.

\section{Compactness for commutators of the class $\{T_\al\}$}

We begin by defining the larger class of bilinear operators $\{T_\al\}$, with $\al$ in some appropriate open interval contained in $\mathbb R$.

Fix $0<\al<2n$ and let $K_\al (x,y,z)$ be a kernel on $\R^{3n}$ defined away from $x=y=z$ that satisfies
\begin{equation}\label{size}|K_\al(x,y,z)|\lesssim \frac{1}{(|x-y|+|x-z|)^{2n-\al}}\end{equation}
and
\begin{equation}\label{smoothness}|K_\al(x,y,z)-K_\al(x+h,y,z)|\lesssim \frac{|h|}{(|x-y|+|x-z|)^{2n-\al+1}},\end{equation}
with the analogous estimates in the $y$ and $z$ variables.  We consider the bilinear operator $T_\al$
\begin{equation}\label{operator}
T_\al (f,g)(x)=\int_{\R^{2n}}K_\al(x,y,z)f(y)g(z)\,dydz\
\end{equation}
defined a priori for, say, $f,g$ bounded and with compact support.
It is easy to see that they extend with the same integral definition \eqref{operator} to bounded operators from
$L^p \times L^q \to L^r$ provided $0<\alpha < 2n, 1 < p,q < \infty,\, \al/n< 1/p+1/q$, and  $1/r = 1/p + 1/q- \al/n$. Clearly,  the analog kernels for $\al=0$ correspond to a  \emph{bilinear Calder\'on-Zygmund kernel}, see again \cite{GT}.

The typical example of the above operators is, of course,  the \emph{bilinear Riesz potential operator} $\I_\al$,  given by the kernel
$$K_\al(x,y,z)=\frac{1}{(|x-y|+|x-z|)^{2n-\al}}.$$
The first relevant observation about the family $\{T_\al\}$ is that, with respect to boundedness, its commutators behave similarly as in the ``end-point'' case $\al=0$.

\begin{theoremc*} Let  $0<\al<2n$, $1<p,q<\infty$, $r\geq 1$, $\frac{\al}{n}<\frac1p+\frac1q $, $\frac1r=\frac1p+\frac1q-\frac{\al}{n}$ and $b
\in BMO$. The following estimates hold:
\begin{align*}
\|[T_\al, b]_1(f,g)\|_{L^r}&\lesssim \|b\|_{BMO}\|f\|_{L^p}\|g\|_{L^q},\\
\|[T_\al, b]_2(f,g)\|_{L^r}&\lesssim \|b\|_{BMO}\|f\|_{L^p}\|g\|_{L^q}.
\end{align*}
\end{theoremc*}

As usual, the notation $x\lesssim y$ indicates that $x\leq Cy$ with a positive constant $C$ independent of $x$ and $y$. For a proof of the above boundedness properties, see the papers by Chen and Xue \cite{CX} and Lian and Wu \cite{LW}. In the linear case the corresponding result goes back to the work of Chanillo \cite{Ch}. The results for the multilinear \cz\, case used in Theorem A were addressed by P\'erez and Torres \cite{PT}, Tang \cite{Tan08}, and Lerner, Ombrosi, P\'erez, Torres and Trujillo-Gonz\'alez \cite{LOPTT}.

Our real interest, however, lies in the possibility of improving boundedness to compactness. In the linear case, the compactness  of the commutators of fractional integrals and multiplication by appropriate functions  has already received some attention
in  several contexts. See, for example, the work of Chen, Ding and Wang \cite{CDW} where the compactness in the
usual Lebesgue measure case is traced back to Wang \cite{W}. See also Betancor and Fari\~na's work \cite{BF} for the setting of non-doubling measures;  the boundedness in this case was  obtained by Chen and Sawyer \cite{CS}.

We note that Theorem B intrinsically assumes that $r\geq 1$. The boundedness
 result in Theorem C
of the operators $[T_\al, b]_1$, $[T_\al, b]_2$  when $r>1$, and  $1/p+1/q<1$, can be alternately obtained as follows. The kernel bound \eqref{size}, implies that
$$
|T_\al(f,g)(x)|\lesssim \int_{\R^{2n}}\frac{|f(y)||g(z)|}{(|x-y|+|x-z|)^{2n-\al}}\,dydz=\I_\al(|f|,|g|)(x).
$$
As shown by Moen \cite{M}, the operator $\I_\al$ satisfies appropriate weighted estimates. Therefore, so does $T_\al$, and we can use the ``Cauchy integral trick''. An exposition of this ``trick'' can be found in the proof of Theorem~\ref{thm:bdd}, which deals with the more singular versions $BI_\al$ of the operators $T_\al$.
Our first main result is an extension of Theorem A that encompasses the commutators of the family $\{T_\al\}_{0<\al< 2n}$.
\begin{theorem}\label{main1}
Let $0< \al<2n$, $1<p,q<\infty$, $1\leq r<\infty$,   $\frac{\al}{n}< \frac1p+\frac1q$,   $\frac1r=\frac1p+\frac1q-\frac{\al}{n},$ and let $b\in CMO$. If $T_\al$ is the bilinear operator defined by \eqref{operator} whose kernel $K_\al$ satisfies \eqref{size} and \eqref{smoothness}, then $[T_\al, b]_1, [T_\al, b]_2: L^p\times L^q\ra L^r$ are compact.
\end{theorem}
\begin{proof} We will work with $[T_\al, b]_1$; by symmetry, the proof for $[T_\al, b]_2$ is the same. By the form of the norm estimates in Theorem C, density, and the results about limits of compact bilinear operators in the operator norm proved in \cite{BT}, we may assume that $b\in C_c^\infty$. Denote by $B_{1,p}$ and $B_{1,q}$ the unit balls in $L^p$ and $L^q$, respectively and let $\mathcal K=[T_\al, b]_1 (B_{1,p},B_{1,q})$. Since $[T_\al, b]_1$ is a bounded operator, see Theorem~C, it is clear that $\mathcal K$ is a bounded set in $L^r$, thus fulfilling condition (a) in Theorem B.  We now aim to show that condition (b) in Theorem B holds.

We introduce the following two indices:
$$\alpha_p = \alpha (1/p+1/q)^{-1} 1/p\,\,\,\, \text{and}\,\,\,\, \alpha_q = \alpha (1/p+1/q)^{-1} 1/q.$$
Clearly, $\alpha_p + \alpha_q = \alpha$. Since $1/p+1/q - \alpha/n >0$,
there exist $s_p>p>1 $ and $s_q>q>1$ such that
$$1/s_p = 1/p - \alpha_p/n\,\,\,\,\text{and}\,\,\,\,
1/s_q = 1/q - \alpha_q/n.$$
Now, since $p, q>1$, we see that $n>\max (\al_p, \al_q)$. In particular, this yields
$$(|x-y| +|x-z|)^{2n-\al} = (|x-y| +|x-z|)^{(n-\al_p)+(n-\al_q)}\geq |x- y|^{n-\al_p} |x-z|^{n-\al_q}.$$
Pick now $R>1$ large enough so that $R>2\max\{|x|: x\in \text{supp}\,b\}$. Using \eqref{size} we see that, for $|x|>R$, we have
\begin{align*}
|[T_\al, b]_1&(f,g)(x)|\lesssim \|b\|_\infty\int_{\R^n}\int_{y\in \text{supp}\,b}\frac{|f(y)||g(z)|}{(|x-y|+|x-z|)^{2n-\al}}\,dydz\\
& \leq \|b\|_\infty \int_{y\in \text{supp} \,b} \int_{\R^n}\frac{|f(y)| |g(z)|}{|x- y|^{n-\al_p} |x-z|^{n-\al_q}}\,dzdy\\
& \lesssim  \frac{\|b\|_\infty}{|x|^{n-\al_p}} \int_{y\in \text{supp} \,b}|f(y)|\int_{\R^n}\frac{|g(z)|}{|x-z|^{n-\al_q}}\,dz \,dy\\
&\lesssim\frac{\|b\|_\infty \, I_{\al_q}(|g|)(x) \|f\|_{L^p}}{|x|^{n-\al_p}} |\text{supp}\,b|^{1/p'}.
\end{align*}
Here, we abused a bit the notation and wrote $I_\al$ also for the linear Riesz potential, $I_\al (f)(x)=\int_{\rn} \frac{f(x)}{|x-y|^{n-\al}}\,dy$. Next, we observe that, since $s_p (n- \al_p)=np > n,$ the function $|x|^{s_p (\al_p-n)}$ is integrable at infinity. Therefore, for a given $\epsilon>0$, we will be able to select an $R=R(\epsilon)$ (but independent of $f$ and $g$) such that
$$\Big(\int_{|x|>R}|x|^{s_p(\al_p-n)}\,dx\Big)^{1/s_p}<\epsilon.$$

Notice now that the indices $s_p, s_q>1$ satisfy $1/r =1/s_p +1/s_q$. Therefore, we can raise the previous point-wise estimate on $|[T_\al, b]_1 (f,g)(x)|$ to the power $r$, integrate over $|x|>R$, and use the H\"older inequality and the $L^q\ra L^{s_q}$ boundedness of $I_{\al_q}$ to get
$$
\Big(\int_{|x|>R} |[T_\al, b]_1(f,g)(x)|^r\,dx\Big)^{1/r}\lesssim \ep\|f\|_{L^p}\|I_{\al_q}(|g|) \|_{L^{s_q}}\lesssim \ep  \|f\|_{L^p} \|g\|_{L^q};
$$
this, in turn, proves that condition (b) in Theorem B is satisfied.

Next, we will use the smoothness of $b$ and that of the kernel $K_\al$ to show that condition (c) in Theorem B holds; specifically, we want to show that
$$\lim_{t\ra 0}\int_{\R^n}|[T_\al, b]_1(f,g)(x+t)-[T_\al, b]_1(f,g)(x)|^r\,dx=0.$$
We use the following splitting from \cite{BT}:
$$[T_\al, b]_1(f,g)(x+t)-[T_\al, b]_1(f,g)(x)=A(x)+B(x)+C(x)+D(x),$$
where, for $\delta>0$ to be chosen later, we have
\begin{align*}
A(x)&=\iint_{|x-y|+|x-z|>\delta}(b(x+t)-b(x))K_\al(x,y,z)f(y)g(z)\,dydz\\
B(x)&=\iint_{|x-y|+|x-z|>\delta}(b(x+t)-b(y))(K_\al(x+t,y,z)-K_\al(x,y,z))f(y)g(z)\,dydz\\
C(x)&=\iint_{|x-y|+|x-z|\leq\delta}(b(y)-b(x))K_\al(x,y,z)f(y)g(z)\,dydz\\
D(x)&=\iint_{|x-y|+|x-z|\leq\delta}(b(x+t)-b(y))K_\al(x+t,y,z)f(y)g(z)\,dydz\\
\end{align*}
The term $A$ is easy to handle with the mean value theorem; we have
$$
|A(x)|\lesssim |t|\|\nabla b\|_\infty \I_\al(|f|,|g|)(x).
$$
Consequently, we obtain
\begin{equation}\label{esti-A}
\|A\|_{L^r}\lesssim |t|\|\nabla b\|_{L^\infty}\|f\|_{L^p}\|g\|_{L^q}.
\end{equation}
{We now consider the terms $B$, $C$ and $D$.}\footnote{We actually obtain estimates for these terms that slightly improve on the corresponding ones for $\al=0$ in \cite{BT}.}
We start with $B$.
\begin{align*}
|B(x)|&\leq\iint_{|x-y|+|x-z|>\delta}(b(x+t)-b(y))(K_\al(x+t,y,z)-K_\al(x,y,z))f(y)g(z)\,dydz\\
&\leq 2\|b\|_\infty\iint_{|x-y|+|x-z|>\delta}|K_\al(x+t,y,z)-K_\al(x,y,z)|\, |f(y)|\,|g(z)|\,dydz\\
&\lesssim |t|\|b\|_\infty\iint_{|x-y|+|x-z|>\delta}\frac{|f(y)||g(z)|}{(|x-y|+|x-z|)^{2n-\al+1}}\,dydz\\
&\lesssim |t|\|b\|_\infty\iint_{\max(|x-y|,|x-z|)>\frac{\delta}{2}}\frac{|f(y)||g(z)|}{\max(|x-y|,|x-z|)^{2n-\al+1}}\,dydz\\
&=|t|\|b\|_\infty\sum^\infty_{k=0}\iint_{2^{k-1}\delta<\max(|x-y|,|x-z|)\leq2^k{\delta}}\frac{|f(y)||g(z)|}{\max(|x-y|,|x-z|)^{2n-\al+1}}\,dydz\\
&\leq|t|\|b\|_\infty\sum^\infty_{k=0}\frac{1}{(2^k\delta)^{2n-\al+1}}\iint_{\max(|x-y|,|x-z|)\leq2^k{\delta}}{|f(y)||g(z)|}\,dydz.
\end{align*}
Note now that
$$\{(y,z)\in \R^{2n}:\max(|x-y|,|x-z|)\leq2^k{\delta}\} \, { \subset} \, B_{2^{k+1}\delta}(x)\times B_{2^{k+1}\delta}(x),$$
where $B_r(x)$ denotes the ball of radius $r$ centered at $x$. Therefore, we can further estimate
\begin{align*}
|B(x)|&\lesssim \frac{|t|\|b\|_\infty}{\delta}\sum^\infty_{k=0}  \frac{ |B_{2^k\delta}(x)|^{\frac{\al}{n}}} {2^{k}} \frac{1}{|B_{2^k\delta}(x)|}\int_{B_{2^k\delta}(x)}  |f(y)|\,dy \frac{1}{|B_{2^k\delta}(x)|}\int_{B_{2^k\delta}(x)}  | g(z)|\,dz\\
&\lesssim |t|\|b\|_\infty \frac{1}{\delta}\M_\al(f,g)(x)\Big(\sum^\infty_{k=0}2^{-k}\Big)=\frac{2|t|\|b\|_\infty}{\delta}\M_\al(f,g)(x),
\end{align*}
where
$$\M_\al(f,g)(x)=\sup_{Q\ni x} |Q|^{\al/n}\Big(\dashint_Q|f(y)|\,dy\Big)\Big(\dashint_Q|g(z)|\,dz\Big).$$
Since the operator $\M_\al(f,g)$ is pointwise smaller than $I_\al(|f|,|g|)$, we get $\M_\al:L^p\times L^q\ra L^r$.  In turn, this yields
\begin{equation}
\label{esti-B}
\|B\|_{L^r}\lesssim \frac{|t|\|b\|_\infty}{\delta}\|f\|_{L^p}\|g\|_{L^q}.
\end{equation}

Let us now estimate the $C$ term.
\begin{align*}
|C(x)|&\leq\iint_{|x-y|+|x-z|\leq\delta}|b(y)-b(x)|\,|K_\al(x,y,z)|\,|f(y)|\,|g(z)|\,dydz\\
&\lesssim \|\nabla b\|_{\infty}\iint_{|x-y|+|x-z|\leq\delta}\frac{|x-y|}{(|x-y|+|x-z|)^{2n-\al}}\,|f(y)||g(z)|\,dydz\\
&\leq \|\nabla b\|_{\infty}\iint_{|x-y|+|x-z|\leq\delta}\frac{|f(y)||g(z)|}{(|x-y|+|x-z|)^{2n-\al-1}}\,\,dydz\\
&\lesssim \|\nabla b\|_{\infty}\sum_{k=0}^\infty \iint_{2^{-k-1}\delta<\max(|x-y|,|x-z|)\leq2^{-k}\delta}\frac{|f(y)||g(z)|}{\max(|x-y|,|x-z|)^{2n-\al-1}}\,\,dydz\\
&\lesssim \|\nabla b\|_{\infty}\sum_{k=0}^\infty \frac{2^{-k}\delta}{(2^{-k}\delta)^{2n-\al}}\iint_{\max(|x-y|,|x-z|)\leq2^{-k}\delta}{|f(y)||g(z)|}\,\,dydz\\
&\lesssim \delta\|\nabla b\|_{\infty}\M_\al(f,g)(x).\\
\end{align*}
From here, we get
\begin{equation}
\label{esti-C}
\|C\|_{L^r}\lesssim \delta\|\nabla b\|_\infty\|f\|_{L^p}\|g\|_{L^q}.
\end{equation}
For the last term $D$ we have an identical estimate to the $C$ term, except that $x$ is now replaced by $x+t$. We have
\begin{align*}
|D(x)|&\leq\iint_{|x-y|+|x-z|\leq\delta}|b(x+t)-b(y)|\,|K_\al(x+t,y,z)f(y)g(z)|\,dydz\\
&\lesssim\|\nabla b\|_{\infty} \iint_{|x+t-y|+|x+t-z|\leq \delta+2|t|}\frac{|x+t-y|  |f(y)|\,|g(z)|}{(|x+t-y|+|x+t-z|)^{2n-\al}}\,dydz\\
&\leq\|\nabla b\|_{\infty} \iint_{|x+t-y|+|x+t-z|\leq \delta+2|t|}\frac{|f(y)|\,|g(z)|}{(|x+t-y|+|x+t-z|)^{2n-\al}}\,dydz\\
&\lesssim (\delta+|t|)\|\nabla b\|_\infty \M_\al(f,g)(x+t).
\end{align*}
Thus, as above, we get
\begin{equation}
\label{esti-D}
\|D\|_{L^r}\lesssim (\delta+|t|)\|\nabla b\|_\infty\|f\|_{L^p}\|g\|_{L^q}.
\end{equation}
Let $1>\epsilon>0$ be given. 
For each $0< |t| < \epsilon^2$ we now select $\delta= |t|/\epsilon$. Estimates ~\eqref{esti-A}, ~\eqref{esti-B}, ~\eqref{esti-C} and \eqref{esti-D} then prove
$$\|[T_\al, b]_1(f, g)(\cdot+t)-[T_\al, b]_1 (f, g)(\cdot)\|_{L^r}\lesssim \epsilon\|f\|_{L^p}\|g\|_{L^q},$$
that is, condition (c) in Theorem B holds.
\end{proof}

\noindent {\bf Remark.} {\it Iterated commutators} can be considered as well. For example, one can look at operators of the form
$$
[T, b_1, b_2](f,g)=[[T, b_1]_1, b_2]_2(f,g)=[T, b_1]_1(f,b_2g)-b_2[T, b_1]_1(f,g).
$$
For bilinear \cz\, operators, the boundedness of such operators was studied in \cite{Tan08}, see also the work by P\'erez, Pradolini, Torres and Trujillo-Gonz\'alez \cite{PPTT}, while for
bilinear fractional integrals they were addressed in  \cite{LW}.  As pointed out in \cite{BT},
the compactness of the iterated commutators is actually easier to prove. The interested reader may adapt the arguments in \cite{BT} to our current situation $\{T_\al\}$.

\section{Separate compactness for commutators of the class $\{BI_\al\}$}
We will now examine a more singular family of bilinear fractional integral operators,
$$BI_\al(f,g)(x)=\int_{\R^n}\frac{f(x-y)g(x+y)}{|y|^{n-\al}}\,dy.$$
These operators were first introduced by Grafakos in \cite{G}, and later studied by Grafakos and Kalton \cite{GK} and Kenig and Stein \cite{KS}. We can view them as fractional versions of the bilinear Hilbert transform
$$BHT(f,g)(x)=p.v.\int_{\R}\frac{f(x-y)g(x+y)}{y}\,dy.$$
For $i=1, 2$ and $b\in BMO$, we define the commutators $[BI_\al, b]_i$ similarly to those of the operators $T_\al$. First, we prove that the commutators $[BI_\al, b]_i$, $i=1, 2$, are bounded. Our proof makes use of what we call the ``Cauchy integral trick''.

\begin{theorem} \label{thm:bdd} Let $0<\al<n$, $1<p, q, r<\infty$, $\frac1p+\frac1q<1$, $\frac1r=\frac{1}{p}+\frac{1}{q}-\frac{\al}{n}$, and $b\in BMO$. Then, for $i=1, 2$, we have
$$\|[BI_\al, b]_i (f, g)\|_{L^r}\leq C\|b\|_{BMO}\|f\|_{L^{p}}\|g\|_{L^{q}}.$$
\end{theorem}
\begin{proof} We will work with the commutator in the first variable; the proof for the second variable is identical. We define $s>1$ by $\frac1s=\frac1p+\frac1q$. As observed by Bernicot, Maldonado, Moen and Naibo \cite{BMMN}, see also \cite{M}, if $1<s<r$ satisfy $\frac1s-\frac1r=\frac{\al}{n}$, then $BI_\al$ is bounded on appropriate product weighted Lebesgue spaces; we have
\begin{equation}\label{weightedbounds}BI_{\al}:L^{p}(w_1^{p})\times L^{q}(w_2^{q})\ra L^r(w_1^{r}w_2^{r})\end{equation}
where $w_1,w_2\in A_{s, r}$, that is, for $i=1, 2$,
$$\sup_Q\left(\frac{1}{|Q|}\int_Q w_i^r\,dx\right)\left(\frac{1}{|Q|}\int_Q w_i^{-s'}\,dx\right)^{\frac{r}{s'}}<\infty.$$
Without loss of generality, we may assume $f, g\in C_c^{\infty}(\R^n)$ and $b$ is real valued.  For $z\in \C$, consider the holomorphic function (in $z$)
$$T_{z} (f, g; \al)=e^{zb}BI_\al(e^{-zb}f, g),$$
and notice that by the Cauchy integral formula, for $\ep>0$,
$$
[BI_\al, b]_1(f, g)=-\frac{d}{dz}T_{z} (f, g; \al)\Big|_{z=0}=-\frac{1}{2\pi i}\int_{|z|=\ep}\frac{T_{z} (f, g; \al)}{z^2}\,dz.
$$
Since $r>1$, we can use Minkowski's integral inequality to obtain
$$\|[BI_\al, b]_1 (f, g)\|_{L^r}\leq \frac{1}{2\pi \ep^2}\int_{|z|=\ep}\|T_{z} (f, g; \al)\|_{L^r}\,|dz|$$
and
$$\|T_{z} (f, g; \al)\|_{L^r}^r=\int_{\R^n}\left( |BI_\al (e^{-zb}f, g)| e^{(\text{Re}\, z) b}\right)^r\,dx.$$
For $\ep>0$, $\ep\lesssim \|b\|_{BMO}^{-1}$, and $|t|\leq \ep$, by John-Nirenberg's inequality, we have $e^{tb}\in A_{s, r}$.  Therefore, by \eqref{weightedbounds} with $w_1=e^{b}$ and $w_2=1$, we have
\begin{align*}
 \|T_{z} (f, g; \al)\|_{L^r}&=\left(\int_{\R^n}\left( |BI_\al(e^{-zb}f, g)|e^{(\text{Re}\, z) b}\right)^r\,dx\right)^{1/r}\\
 &\leq C\left(\int_{\R^n}(|e^{-zb}f|e^{(\text{Re}\, z) b})^{p}\,dx\right)^{1/p}\left(\int_{\R^n}|g|^{q}\,dx\right)^{1/q}\\
 &=C\|f\|_{L^{p}}\|g\|_{L^{q}}.
 \end{align*}
The desired result follows from here.
\end{proof}

\begin{theorem}\label{main2} Let $0< \al<n$, $1<p, q, r<\infty$, $\frac{1}{p}+\frac{1}{q}<1$, $\frac1r=\frac{1}{p}+\frac{1}{q}-\frac{\al}{n}$, and $b\in CMO$. Then, $[ BI_\al, b]_1, [BI_\al,b]_2: L^p\times L^q\ra L^r$ are separately compact.
\end{theorem}

\begin{proof} We will work again with the commutator in the first variable.  By a change of variables, this commutator can be rewritten as
$$[BI_\al, b]_1(f,g)(x)=\int_{\R^n}\frac{b(y)-b(x)}{|x-y|^{n-\al}}f(y)g(2x-y)\,dy.$$
We may assume that $b\in C_c^\infty(\R^n)$ and aim to prove that the conditions (a), (b) and (c) of Theorem~B hold for the family of functions $[BI_\al, b]_1(f,g)$, where $g\in L^q$ is fixed and $f\in B_{1, p}$.

By Theorem~\ref{thm:bdd}, we already know that condition (a) is satisfied. Thus, we concentrate on proving (b) and (c).

The estimates that yield (b) are reminiscent of the ones used in the proof of Theorem~\ref{main1}. Assume $R>1$ is large enough so that $|x|\geq R$ implies $x\notin$ supp $b$.

Then
\begin{align*}|[BI_\al, b]_1(f,g)(x)|&\leq \|b\|_{\infty} \int_{\text{supp}\, b}{|x-y|^{\al-n}}|f(y)g(2x-y)|\,dy\\
&\lesssim \|b\|_{\infty} |x|^{\al-n}\int_{\text{supp}\, b}|f(y)g(2x-y)|\,dy\\
&\leq \|b\|_{\infty} |x|^{\al-n}\left(\int_{\text{supp} \, b} |f(y)|^{q'}\,dy\right)^{1/q'}\|g\|_{L^{q}}\\
\end{align*}
Let us write $\frac1s=\frac1p+\frac1q<1= \frac{1}{q}+\frac{1}{q'}$; so $q'<p$. As such, we can further estimate
$$\left(\int_{\text{supp} \ b} |f(y)|^{q'}\,dy\right)^{1/q'}\leq |\text{supp} \ b|^{\frac{1}{q'}-\frac{1}{p}} \|f\|_{L^{p}}= |\text{supp} \ b|^{\frac{1}{s'}} \|f\|_{L^{p}}.$$
Now, we raise to the power $r$ and integrate with respect to $x$ over the set $|x|>R$.  Notice that, since $s>1$, we have
$\frac1r=\frac1s-\frac{\al}{n}<\frac{n-\al}{n} \Leftrightarrow r(n-\al)>n.$ This allows us, for a given $\ep>0$, to control
$$\int_{|x|>R}|[BI_\al, b]_1 (f,g)(x)|^r\,dx<\ep^r$$
by taking $R=R(\ep)>0$ sufficiently large; which shows that, indeed, (b) is satisfied.

We are left to show the continuity condition (c), that is,
$$\lim_{t\ra 0}\|[BI_\al, b]_1(f,g)(\cdot+t)-[BI_\al, b]_1(f,g)\|_{L^r}=0,$$
uniformly for $\|f\|_{L^{p}}\leq 1$ and $g\in L^q$ fixed.  First, we lump our fixed function $g$ into a general kernel
$$[BI_\al, b]_1(f,g)(x)=\int_{\R^n}\frac{b(y)-b(x)}{|x-y|^{n-\al}}f(y)g(2x-y)\,dy$$$$=\int_{\R^n}(b(y)-b(x))K_g(x,y)f(y)\,dy$$
where
$$K_g(x,y)=\frac{g(2x-y)}{|x-y|^{n-\al}}.$$
Second, we split the commutator $[BI_\al, b]_1$ by following the decomposition used for $[T_\al, b]_1$. Namely, we write $$[BI_\al, b]_1 (f, g)(x+t)-[BI_\al,b]_1(f,g)(x)=A(x)+B(x)+C(x)+D(x),$$ where
\begin{align*}A(x)&=\int_{|x-y|>\delta}(b(x+t)-b(x))K_g(x,y)f(y)\,dy,\\
B(x)&=\int_{|x-y|>\delta}(b(x+t)-b(y))(K_g(x+t,y)-K_g(x,y))f(y)\,dy,\\
C(x)&=\int_{|x-y|\leq\delta}(b(y)-b(x))K_g(x,y)f(y)\,dy,\\
D(x)&=\int_{|x-y|\leq\delta}(b(x+t)-b(y))K_g(x+t,y)f(y)\,dy.\\
\end{align*}
We will now estimate each term in this decomposition. For $A$, the estimate is immediate. We clearly have $|A(x)|\leq |t|\|\nabla b\|_\infty BI_\al(|f|, |g|)(x).$ Since $BI_\al$ is $L^p\times L^q\ra L^r$ bounded, we get $\|A\|_{L^r}\lesssim |t|\|f\|_{L^{p}}\|g\|_{L^{q}}.$

The estimate for the $B$ term is the most delicate. To facilitate the ease of reading, we postpone it until the end of the proof.

We estimate $C$ as follows:
\begin{align*}
|C(x)|&\leq\int_{|x-y|\leq\delta}|b(y)-b(x)|\frac{|g(2x-y)|}{|x-y|^{n-\al}}|f(y)|\,dy\\
&\leq \|\nabla b\|_{\infty}\int_{|x-y|\leq\delta}|x-y|\frac{|g(2x-y)|}{|x-y|^{n-\al}}|f(y)|\,dy\\
&\leq \delta\|\nabla b\|_{\infty} BM_\al(f, g)(x),
\end{align*}
where $BM_\al$ is the associated bilinear fractional maximal operator,
$$BM_\al(f,g)(x)=\sup_{r>0}\frac{1}{r^{n-\al}}\int_{|y|<r}|f(x-y)g(x+y)|\,dy.$$

The estimate for $D(x)$ is similar; we now have
$$|D(x)|\leq (\delta+|t|)\|\nabla b\|_\infty BM_\al(f, g)(x+t).$$
Again, since $BM_\al(f,g)\lesssim BI_\al(|f|,|g|)$, we have $BM_\al:L^p\times L^q\ra L^r$. Thus, similarly to $A$, we get $\|C\|_{L^r}\lesssim \delta\|f\|_{L^{p}}\|g\|_{L^{q}}$ and $\|D\|_{L^r}\lesssim (\delta+|t|)\|f\|_{L^{p}}\|g\|_{L^{q}}.$

Finally, we turn our attention to $B$.
\begin{align*}
|B(x)|&\leq\int_{|x-y|>\delta}|b(x+t)-b(y)|\left|\frac{g(2x+2t-y)}{|x+t-y|^{n-\al}}-\frac{g(2x-y)}{|x-y|^{n-\al}}\right||f(y)|\,dy\\
&\leq 2\| b\|_{\infty}\int_{|x-y|>\delta}\left|\frac{g(2x+2t-y)}{|x+t-y|^{n-\al}}-\frac{g(2x-y)}{|x-y|^{n-\al}}\right||f(y)|\,dy\\
&\lesssim \int_{|x-y|>\delta}\left|\frac{1}{|x+t-y|^{n-\al}}-\frac{1}{|x-y|^{n-\al}}\right||g(2x+2t-y)f(y)|\,dy\\
&\qquad +\int_{|x-y|>\delta}\frac{|g(2x+2t-y)-g(2x-y)||f(y)|}{|x-y|^{n-\al}}\,dy\\
&=E(x)+F(x)
\end{align*}
To estimate $E$, we note that
$$\left|\frac{1}{|x+t-y|^{n-\al}}-\frac{1}{|x-y|^{n-\al}}\right|\lesssim \frac{|t|}{|x-y|^{n-\al+1}},$$
which implies
\begin{align*}
E(x)&\lesssim |t|\int_{|x-y|>\delta}\frac{|g(2x+2t-y)f(y)|}{|x-y|^{n-\al+1}}\,dy\\
& \lesssim \frac{|t|}{\delta} BM_\al(f, \tau_{2t}g)(x).
\end{align*}
Here, $\tau_a$ is the shift operator $\tau_ag(x)=g(x+a)$.  It follows from the boundedness of $BM_\al$ that
$$\|E\|_{L^r}\lesssim \frac{|t|}{\delta}\|f\|_{L^{p}}\|g\|_{L^{q}}.$$
For $F(x)$ we have
$$F(x)\lesssim BM_\al(f, \tau_{2t}g-g)(x),$$
so
$$\|F\|_{L^r}\lesssim \|f\|_{L^{p}} \|\tau_{2t}g-g\|_{L^{q}}.$$
Since $g\in L^{q}$, for a given $\ep>0$ we can find $\gamma=\gamma(\ep,g)>0$ such that $|t|<\gamma$ implies
$$\|\tau_{2t}g-g\|_{L^{q}}<\ep.$$
Finally, by choosing $|t|<\ep^2$ and $\delta =|t|/\ep$ we get that
$$\|[BI_\al, b]_1(f,g)(\cdot+t)-[BI_\al, b]_1(f,g)\|_{L^r}\lesssim \ep.$$
This shows that (c) holds, thus finishing our proof for the compactness in the first variable.

We now show that $[BI_\al,b]_1$ is compact in the second variable, that is, $[BI_\al,b]_1(f,\cdot): L^q\ra L^r$ is compact for a fixed $f\in L^p$.  Conditions (a) and (b) of Theorem B follow from similar calculations to those performed above.  Thus we will check condition (c) of Theorem B.   For $f\in L^p$ fixed and $g\in B_{1,q}$  we write
\begin{align*}
\lefteqn{[BI_\al,b]_1(f,g)(x+t)-[BI_\al,b]_1(f,g)(x)}\\
&=\int_{\R^n}(b(2x+2t-y)-b(x+t))\frac{f(2x+2t-y)g(y)}{|x+t-y|^{n-\al}}\,dy\\
&\qquad -\int_{\R^n}(b(2x-y)-b(x))\frac{f(2x-y)g(y)}{|x-y|^{n-\al}}\,dy\\
&=\int_{\R^n}(b(2x+2t-y)-b(x+t))K_f(x+t,y)g(y)\,dy\\
&\qquad -\int_{\R^n}(b(2x-y)-b(x))K_f(x,y)g(y)\,dy\end{align*}
where this time we combine $f$ with the kernel:
$$K_f(x,y)=\frac{f(2x-y)}{|x-y|^{n-\al}}.$$
Before proceeding further, we make one reduction.  Notice that
\begin{align}
[b,BI_\al]_1(f,g)(x+t)&=\int_{\R^n}(b(2x+2t-y)-b(x+t))K_f(x+t,y)g(y)\,dy\nonumber\\
\label{firstterm}&=\int_{\R^n}(b(2x+2t-y)-b(2x-y))K_f(x+t,y)g(y)\,dy\\
&\qquad +\int_{\R^n}b(2x-y)-b(x+t))K_f(x+t,y)g(y)\,dy\nonumber
\end{align}
The first term in the sum \eqref{firstterm} is bounded by
$$ 2\|\nabla b\|_\infty |t|BI_\al(|f|,|g|)(x+t)$$
and the $L^r$ norm of this quantity will go to zero uniformly for $g\in B_{1,q}$ as $t\ra 0$.   Thus it remains to estimate \begin{multline*}
\int_{\R^n}b(2x-y)-b(x+t))K_f(x+t,y)g(y)\,dy \\-\int_{\R^n}(b(2x-y)-b(x))K_f(x,y)g(y)\,dy
=G(x)+H(x)+I(x)+J(x)
\end{multline*}
where
\begin{align*}
G(x)&=\int_{|x-y|>\delta}(b(x)-b(x+t))K_f(x,y)g(y)\,dy,\\
H(x)&=\int_{|x-y|>\delta}(b(2x-y)-b(x+t))(K_f(x+t,y)-K_f(x,y))g(y)\,dy,\\
I(x)&=\int_{|x-y|\leq \delta}(b(x)-b(2x-y))K_f(x,y)g(y)\,dy,\\
J(x)&=\int_{|x-y|\leq \delta}(b(2x-y)-b(x+t))K_f(x+t,y)g(y)\,dy.
\end{align*}
The estimates for $G,H,I,$ and $J$ are handled similarly to the corresponding estimates for $A,B,C,$ and $D$ above, again, with $H$ being the most complicated.  For example the estimates for $G,I,$ and $J$ are as follows:
$$|G(x)|\leq |t|\|\nabla b\|_\infty BI_\al(f,g)(x),$$
$$|I(x)|\leq \delta\|\nabla b\|_\infty BM_\al(f,g)(x),$$
and
$$|J(x)|\leq (\delta+|t|)\|\nabla b\|_\infty BM_\al(f,g)(x+t).$$
Finally, for $H$ we have
$$
|H(x)|\lesssim \|b\|_\infty\Big( \frac{|t|}{\delta}BM_{\al}(\tau_{2t}f,g)(x)+\frac{1}{\delta}BM_{\al}(\tau_{2t}f-f,g)(x)\Big).
$$
These estimates show that $[BI_\al,b]_1(f,g)$ is compact in the second variable as well, thus showing that it is separately compact.
\end{proof}

A close inspection of the proof of Theorem~\ref{main2} shows that we barely miss proving joint compactness. Indeed, the only non-uniform estimate concerns the very last terms, which we denote by $F$  and $H$, where we use the fact that we can make the quantity
$\|\tau_{2t}g-g\|_{L^{q}}$ (or $\|\tau_{2t}f-f\|_{L^p}$) small by taking $t$ sufficiently small and, crucially, dependent on $g$ (or $f$).  Thus, our method of proof only yields separate compactness.  Compared to the nicely behaved operators $T_\al$, we have in effect a weaker smoothing property of the commutators of the more singular bilinear fractional integrals, $BI_\al$.

The remarks above motivate the following question.

\begin{question} For $b\in CMO$, are the commutators $[BI_\al, b]_i, i=1,2,$ jointly compact? \end{question}

The techniques used in this section can be applied to commutators of the $BHT$, if, {\it a priori}, we know that its commutators are bounded.  Specifically, if we assume that $[BHT,b]_1:L^p\times L^q\ra L^r$, then $[BHT,b]_1$ is separately compact for $b\in CMO$; a similar result holds for $[BHT,b]_2$.  This leads to the following natural question about the bilinear Hilbert transform.

\begin{question} For $b\in BMO$, are the commutators $[BHT,b]_i, i=1,2,$ bounded from $L^p\times L^q\rightarrow L^r$? \end{question}

\end{document}